\documentclass[reqno]{amsart}
\usepackage{amssymb}
\usepackage{fullpage}
\usepackage{amsthm}
\usepackage{bm}
\usepackage{latexsym}
\usepackage{float}
\restylefloat{table}
\usepackage{amsmath}
\usepackage{mathrsfs}
\usepackage[font={small,it}]{caption}
\usepackage{amscd}
\usepackage[all,cmtip]{xy}
\usepackage[usenames]{color}
\usepackage{graphicx}
\usepackage{color}
\usepackage{amscd}
\usepackage{float}
\usepackage{graphics}
\usepackage{tikz}
\usepackage{tikz-cd}
\usepackage{comment}
\usepackage{xspace}
\usepackage{mathtools}
\usepackage{amssymb}
\usepackage{amsthm}
\usepackage{graphicx}
\usepackage{subfig}
\usepackage{enumerate}

\usetikzlibrary{patterns,decorations.pathreplacing}
\usepackage{marginnote}

\theoremstyle{plain}

\newtheorem{theorem}{Theorem}[section]

\newtheorem{lemma}[theorem]{Lemma}
\newtheorem{corollary}[theorem]{Corollary}

\DeclareFontFamily{U}{wncy}{}
    \DeclareFontShape{U}{wncy}{m}{n}{<->wncyr10}{}
    \DeclareSymbolFont{mcy}{U}{wncy}{m}{n}
    \DeclareMathSymbol{\Sh}{\mathord}{mcy}{"58}
\theoremstyle{definition}

\newcommand{\appsection}[1]{\let\oldthesection\thesection
\renewcommand{\thesection}{Appendix \oldthesection}
\section{#1}\let\thesection\oldthesection}

\theoremstyle{remark}

\newtheorem{remark}[theorem]{Remark}
\newtheorem{example}[theorem]{Example}

\def\Z{{\mathbb{Z}}}

\def\Q{{\mathbb{Q}}}
\def\C{{\mathbb{C}}}
\def\P{{\mathbb{P}}}

\usepackage{microtype}

\DeclareMathOperator{\End}{End}

\DeclareMathOperator{\Poles}{Poles}

\begin{document}
\title{Linear $x$-coordinate relations of triples on elliptic curves}

\author{Jerson Caro}
\address{Department of Mathematics \& Statistics, Boston University, 665 Commonwealth Avenue, Boston, MA 02215, USA}
\email[J. Caro]{jlcaro@bu.edu}%

\author{Natalia Garcia-Fritz}
\address{Departamento de Matem\'aticas,
Pontificia Universidad Cat\'olica de Chile.
Facultad de Matem\'aticas,
4860 Av.\ Vicu\~na Mackenna,
Macul, RM, Chile}
\email[N. Garcia-Fritz]{natalia.garcia@uc.cl}%

\thanks{J.C. was supported by a Simons Foundation grant (Grant \#550023). N. G.-F. was supported by ANID Fondecyt regular grant 1211004.}

\date{\today}

\begin{abstract} For an elliptic curve $E$ defined over the field $\C$ of complex numbers, we classify all translates of elliptic curves in $E^3$ such that the $x$-coordinates satisfy a linear equation. This classification enables us to establish a relation between the rank of finite rank subgroups of $E$ and triples in $E$ whose $x$-coordinates are linearly related.  
The method of proof integrates complex analytic techniques on elliptic curves with results of Gao, Ge and Kühne on Uniform Mordell-Lang Conjecture for subvarieties in abelian varieties.
\end{abstract}
\maketitle

\section{Introduction}

Let $E$ be an elliptic curve defined over the field $\C$ of complex numbers, given by the equation 
\begin{equation}\label{Weierstrass}
y^2=x^3+Ax^2+Bx+C,    
\end{equation}
and with $e_E$ its point at infinity. Given a triple $(c_1,c_2,c_3)$ of non-zero complex numbers, we are interested in bounding the number of solutions of the equation 
\begin{equation}\label{eq123}
c_1x_1+c_2x_2+c_3x_3=0
\end{equation} 
with $x_1,x_2,x_3$ in the $x$-coordinates of points, different from $e_E$, on finite rank subgroups $\Gamma<E(\C)$. From now on, the $x$-coordinate of a point $P\in E(\mathbb{C})\setminus\{e_E\}$ is denoted by $x(P)$. 

The main results we prove in this manuscript are Theorem \ref{MainTheorem} and Theorem \ref{arithmain}, that give us as a consequence such a bound depending only on the rank of $\Gamma$: 

\begin{corollary}\label{cor1}
There exists an absolute constant $D$ such that the following holds: 

Let $E/\C$ be a non-CM elliptic curve, and let $(c_1,c_2,c_3)$ be a triple of non-zero complex numbers such that no subsum of $c_1+c_2+c_3$ is equal to zero. Let $\Gamma$ be a subgroup of $E(\mathbb{C})$ of finite rank $r$. 
The number of triples $(P_1,P_2,P_3)\in (\Gamma\setminus\{e_E\})^3$ satisfying $$c_1x(P_1)+c_2x(P_2)+c_3x(P_3)=0$$ is at most $D^{r+1}$. 
\end{corollary}

This can be seen as the elliptic curve analogue (for $n=3$) of Theorem 1.1 in \cite{ESS}, where Evertse, Schlickewei, and Schmidt bound the number of non-trivial solutions of equations of the form $c_1x_1+\cdots+c_nx_n=0$ for finite rank subgroups of $(k^*)^n$ with $k$ an algebraically closed field of characteristic zero (see also \cite{EV}).
Both our Corollary \ref{cor1} as well as Theorem 1.1 in \cite{ESS} are instances (for elliptic curves and the multiplicative group $\mathbb{G}_m$ respectively) of the fact that the group operation of an algebraic group is, in some sense, not compatible with additive relations, unless the group itself is additive.

A version of Corollary \ref{cor1} for CM elliptic curves, removing infinitely many \emph{trivial} solutions, is also a consequence of Theorem \ref{arithmain}, as we are able to give a complete and explicit characterization of the trivial solutions of Equation \eqref{eq123}, for any  elliptic curve $E$. 

The method of proof is the following: From Equation \eqref{eq123} one can define a complex surface $X_{c_1,c_2,c_3}\subseteq E^3$ whose points correspond to either triples $(P_1,P_2,P_3)\in E(\mathbb{C})^3$ whose $x$-coordinates solve Equation \eqref{eq123}, or are triples containing $e_E$. From the Uniform Mordell-Lang conjecture, recently proved by Gao, Ge, and K\"uhne  \cite{GGK}, to bound the number of solutions of Equation \eqref{eq123} we need to explicitly find all translates of abelian subvarieties of $E^3$ in $X_{c_1,c_2,c_3}$, which is a problem equivalent to finding the union of all the images of non-constant morphisms $E\to X_{c_1,c_2,c_3}$. This relies on the fact that every proper abelian subvariety of $E^3$ is isogenous to $E$ or $E^2$.

We give a complete classification of these maps in the following result:

\begin{theorem}\label{MainTheorem}
Let $E$ be an elliptic curve defined over $\C$ and let $(c_1,c_2,c_3)$ be a triple of non-zero complex numbers. Then, the images of all the non-constant morphisms $\Phi\colon E\to  X_{c_1,c_2,c_3}\subset E^3$ are classified into the following cases:
\begin{itemize}
    \item [(a)] $\{(P,e_E,e_E)\colon P\in E\}$, and every coordinate permutation of this elliptic curve.
    \item[(b)] $\{(P,[u]P,R)\colon P\in E\}$ with $R\in E$ and $u$ a unit in $\End(E)$, satisfying respectively $x(R)=-\frac{A(c_1+c_2)}{3c_3}$ and $c_1+u^2c_2=0$, and every coordinate permutation of this scenario, 
    \item[(c)] $\{(P,[u]P,[v]P)\colon P\in E\}$ where $u,v$ are units in $\End(E)$ satisfying $c_1+\frac{c_2}{u^2}+\frac{c_3}{v^2}=0$, and
    \item[(d)] $\{([\sqrt{-2}]P,\pm P,\pm P+(a,0))\colon P\in E\}$, whenever $E:y^2=x^3+ax^2-3a^2x+a^3$ for any non-zero $a\in \C^{\times}$ and $c_1=2c_2=2c_3$, and every coordinate permutation of this scenario.
\end{itemize}
Here, we identify $\End(E)$ with a subring of $\C$.
\end{theorem}

Applying Uniform Mordell-Lang, we obtain a bound for the number of points of a finite rank group outside the Kawamata locus, which we denote by $Z_{c_1,c_2,c_3}$. The result is the following:

\begin{theorem}\label{arithmain}
There exists an absolute constant $D$ such that for any elliptic curve $E/\C$, triple $(c_1,c_2,c_3)$ of non-zero complex numbers, and any finite rank subgroup $\Gamma$ of $E^3(\C)$ the following holds: 
$$\# (X_{c_1,c_2,c_3}\setminus Z_{c_1,c_2,c_3})(\C)\cap\Gamma\leq D^{\mathrm{rk}(\Gamma)+1}.$$
\end{theorem}

When $E$ and $\Gamma$ are defined over $\Q^{alg}$, one can also give effective bounds, depending on $\mathrm{rk}(\Gamma)$, $\mathrm{deg}_LX_{c_1,c_2,c_3}$ with $L$ a symmetric ample sheaf on $E^3$, and the Faltings height of $E^3$, by Th\'eor\`eme 1.3 in R\'emond's work \cite{Rem00}.

\subsection*{Acknowledgments}
We would like to thank Hector Pasten for fruitful discussions.

J.C. was supported by a Simons Foundation grant (Grant 550023). N. G.-F. was supported by ANID Fondecyt regular grant 1211004.

\section{Finding the Kawamata locus of $X_{c_1,c_2,c_3}$}

\subsection{Geometry of $X_{c_1,c_2,c_3}$ and Theorem \ref{MainTheorem} case (a)} Given $(c_1,c_2,c_3)\in(\mathbb{C}^\times)^3$, consider the surface $Y_{c_1,c_2,c_3}=\mathbb{V}(c_1x_1+c_2x_2+c_3x_3)\subseteq\mathbb{A}^3$. We can view $Y_{c_1,c_2,c_3}$ as a quasiprojective variety in $(\mathbb{P}^1)^3$ via the identification 
\begin{eqnarray*}
\mathbb{A}^3 &\to&(\mathbb{P}^1)^3\cr
(x_1,x_2,x_3) &\mapsto& ([x_1:1],[x_2:1],[x_3:1]),
\end{eqnarray*}
and we denote by $\bar{Y}_{c_1,c_2,c_3}\subseteq(\mathbb{P}^1)^3$ the Zariski closure of the image of $Y_{c_1,c_2,c_3}$.

Let $E/\C$ be the elliptic curve defined by the Weierstrass equation $y^2=x^3+Ax^2+Bx+C$, and denote its point at infinity $[0:1:0]$ by $e_E$. Consider the $x$-coordinate map 
\begin{eqnarray*}
E &\to&\mathbb{P}^1\cr
[x:y:1] &\mapsto& [x:1]\cr
[0:1:0] &\mapsto& [1:0],
\end{eqnarray*}
which induces a natural map $F\colon E^3\to(\mathbb{P}^1)^3$. Let $X_{c_1,c_2,c_3}:=F^{-1}(\bar{Y}_{c_1,c_2,c_3})\subseteq E^3$. As in Subsection 5.3 of \cite{GFP}, one can directly prove that
\begin{lemma}
The set $X_{c_1,c_2,c_3}\subseteq E^3$ is a projective surface, and $F_{|X_{c_1,c_2,c_3}}\colon X_{c_1,c_2,c_3}\to(\mathbb{P}^1)^3$ is a surjective, flat, and finite morphism of degree $8$.
\end{lemma}

The following lemma is also easy to verify:

\begin{lemma}
Let $a,b,c\in\mathbb{C}$. We have that $a,b,c$ satisfy $c_1a+c_2b+c_3c=0$ if and only if $(a,b,c)\in Y_{c_1,c_2,c_3}$. 

Let $P_1,P_2,P_3\in E(\mathbb{C})$, none of them equal to $e_E$. We have that $x(P_1),x(P_2),x(P_3)$ satisfy $$c_1x(P_1)+c_2x(P_2)+c_3x(P_3)=0$$ if and only if $(P_1,P_2,P_3)\in X_{c_1,c_2,c_3}\cap (E\setminus\{e_E\})^3$.
\end{lemma}

Let $\Phi\colon E\to X_{c_1,c_2,c_3}\subset E^3$ be a non-constant morphism. Considering the previous Lemma, we first study $\mathrm{Im}(\Phi)$ when $\Phi$ is identically $e_E$ in at least one component.

\begin{lemma}\label{Lemma}
Let $c_1,c_2,c_3$ be non-zero complex numbers and let $\Phi\colon E\to X_{c_1,c_2,c_3}\subseteq E^3$ be a non-constant morphism of varieties that is identically $e_E$ in at least one component. Then $\Phi(E)=E\times\{e_E\}\times \{e_E\}$ or some coordinate permutation of this elliptic curve.

\end{lemma}

\begin{proof}
Let us denote by $x_i$ and $y_i$ the local coordinates of the $i$-th component of $(\P^1)^3$ and consider the open set $\prod_{i=1}^3\{x_i\neq 0\}$ in $(\P^1)^3$. In this open set, Equation \eqref{eq123} turns into 
$$
c_1y_2y_3+c_2y_1y_3+c_3y_1y_2=0.
$$
Consequently, if a point $(z_1,z_2,z_3) \in \bar{Y}_{c_1,c_2,c_3}$ has some coordinate $z_j=\infty=[1:0]\in \P^1$, then another coordinate is also equal to $\infty$ and the remaining coordinate is free. This implies that $\bar{Y}_{c_1,c_2,c_3}\setminus Y_{c_1,c_2,c_3}$ consists of three lines: $\P^1\times \{\infty\}\times \{\infty\}$ and every coordinate permutation of this line.

As we assume that $\Phi$ is identically $e_E$ in one component, we have that $F(\Phi(E))\subset\bar{Y}_{c_1,c_2,c_3}\setminus Y_{c_1,c_2,c_3}$. Since $\Phi$ is non-constant, $F(\Phi(E))$ must be equal to $\P^1\times \{\infty\}\times \{\infty\}$ or some coordinate permutation of this line. Finally, by the definition of $F$ we obtain that $\Phi(E)=E\times\{e_E\}\times \{e_E\}$ or some coordinate permutation of this elliptic curve.
\end{proof}

\subsection{Theorem \ref{MainTheorem} case (b)}

In the following, we fix a triple $(c_1,c_2,c_3)$ of non-zero complex numbers and a non-constant morphism of varieties $\Phi\colon E\to X_{c_1,c_2,c_3}\subseteq E^3$, where none of the components of $\Phi$ is identically $e_E$. 
The change of variables $y\mapsto y/2$ and $x\mapsto x-A/3$ turns Equation \eqref{Weierstrass} into
\begin{equation}\label{Weierstrass2}
y^2=4x^3+bx+c    
\end{equation}
for some $b,c\in \C$. Therefore, the $x$-coordinate of $E$ is locally given by $\wp(z)+A/3$, where $\wp$ is the Weierstrass elliptic function associated to Equation \eqref{Weierstrass2}. Let $\Lambda=\langle 1,\tau\rangle$ be the lattice associated to $E$.

Taking covering spaces, we have the following commutative diagram:
\[
\xymatrixcolsep{9pc}\xymatrix{
\C\ar[d]\ar[r]^{(\alpha_i z+\beta_i)_i} & \C^3 \ar[d] \ar^{(\wp+\frac{A}{3})_i}[dr] & \\
E \ar[r]_{\Phi} & E^3\ar_{(x_1,x_2,x_3)}[r] & (\P^1)^3},
\]
where $\alpha_i\in\End(E)$. Thus, we are looking for $\alpha_i,\beta_i$ such that the following equation holds:
\begin{equation}\label{eq1}
c_1\wp(\alpha_1 z+c_2\beta_1)+c_2\wp(\alpha_2 z+c_2\beta_2)+c_3\wp(\alpha_3 z+\beta_3)=-\frac{A}{3}(c_1+c_2+c_3).
\end{equation}

Note that for a non-constant $\Phi$ it is not feasible that only one $\alpha_i$ be non-zero. If that were the case, we would have that $\wp(\alpha_iz+\beta_i)$ is constant, which implies that $\alpha_i=0$. 
Hence, it is enough to study Equation \eqref{eq1} in two different cases: (i) one $\alpha_i$ is zero, and (ii) every $\alpha_i$ is nonzero. Case (i) is covered by the following:

\begin{lemma}\label{Lemma 1}
Let $\Phi\colon E\to E^3$ be a non-constant morphism with no component being identically $e_E$ such that $\Phi(E)\subset X_{c_1,c_2,c_3}$ and $\Phi$ is constant in one component. Then, there exist $R\in E$ satisfying $x(R)=-\frac{A(c_1+c_2)}{3c_3}$ and $u\in \End(E)^{\times}$ satisfying $c_1+u^2c_2=0$ such that 
$\Phi(E)=\{(P,[u] P,R)\colon P\in E\}$,
or a coordinate permutation of this scenario.
\end{lemma}
\begin{proof}
Without loss of generality, let us assume that $\Phi$ is constant on the third component. Then, Equation \eqref{eq1} becomes:
\begin{equation}\label{eq2}
c_1\wp(\alpha_1 z+\beta_1)+c_2\wp(\alpha_2 z+\beta_2)+c_3\wp(\beta_3)=-\frac{A}{3}(c_1+c_2+c_3).    
\end{equation}
This implies that $\Poles(\wp(\alpha_1 z+\beta_1))=\Poles(\wp(\alpha_2 z+\beta_2))$. Consequently, we have $(1/\alpha_1)\Lambda=(1/\alpha_2)\Lambda$, which is equivalent to stating that $\alpha_2=u\alpha_1$ for some $u\in \End(E)$. 
By computing the coefficient of $z^{-2}$ in Equation \eqref{eq2}, we obtain:
\[
\frac{c_1}{\alpha_1^2}+\frac{c_2}{u^2\alpha_1^2}=0,
\]
which implies $c_1u^2+c_2=0$. Since the Laurent expansion of $\wp$ does not have constant term, we obtain the following equation
\[
c_3\wp(\beta_3)=-\frac{A}{3}(c_1+c_2+c_3).
\]
Let $R\in E$ be the associated point to $\beta_3$ modulo $\Lambda$, then we have that
\[
x(R)=\wp(\beta_3)+\frac{A}{3}=-\frac{A(c_1+c_2)}{3c_3},
\]
and the image of $\Phi$ is of the curve $\{(P,[u] P,R)\colon P\in E\}$ which is the translate of an elliptic curve by $(e_E,e_E,R)$.
\end{proof}

\subsection{Proof of Theorem \ref{MainTheorem} cases (c) and (d)}\label{sectionmain}
In this subsection, we will concentrate on case where $\Phi$ maps onto each projection, that is, $\alpha_1\alpha_2\alpha_3\neq 0$.

Equation \eqref{eq1} implies that 
\begin{equation}\label{eq3}
\Poles(\wp(\alpha_i z+\beta_i))\subset \bigcup_{j\neq i} \Poles(\wp(\alpha_j z+\beta_j)).    
\end{equation}
We will begin by classifying the morphisms of varieties $\Phi\colon E\to E^3$ which map onto every coordinate and satisfy Equation \eqref{eq3}, but not necessarily mapping into $X_{c_1,c_2,c_3}$.

\begin{lemma}\label{Prop1}
Assume that there exist non trivial $\alpha_1,\alpha_2,\alpha_3\in\mathrm{End}(E)$ and $\beta_1,\beta_2,\beta_3\in\mathbb{C}$ such that Equation \eqref{eq3} holds.
If 
\begin{equation}\label{origen}
\bigcap_{j= 1}^3 \Poles(\wp(\alpha_j z+\beta_j))=\emptyset,
\end{equation} 
then $2$ is reducible in $\End(E)$ and there exist a unit $u$ of $\End(E)$ and an endomorphism $\lambda$ of degree $2$ such that $\alpha_j=u\alpha_i$ and $\alpha_k=\lambda\alpha_i$.
\end{lemma}

\begin{remark}
The Equation \eqref{origen} is equivalent to state that the associated morphism $\Phi$ does not pass through the origin.
\end{remark}

Let us provide an example satisfying Equation \ref{eq3} before proving Lemma \ref{Prop1}:

\begin{example}\label{ejemplo25}
Consider the elliptic curve $E$ given by the equation $y^2=x^3-x$. This curve has complex multiplication with $\End(E)\cong\Z[i]$ and $i(x,y)=(-x,iy)$. Define the morphism $\Phi\colon E\to E^3$ by
\[
\Phi(P)=([1+i]P,P,P-(0,0)).
\]
Since $(1+i)(0,0)=e_E$, the map $\Phi$ satisfies Equation \eqref{eq3}.
\end{example}

Lemma \ref{Prop1} implies that Example \ref{ejemplo25} is the only kind of morphism that does not pass through the origin and  satisfies Equation \eqref{eq3}. First, we prove a result that will be useful in the proof of Lemma \ref{Prop1}.

\begin{lemma}\label{Lemma 2}
Let $G$ be an abelian group and let $H$ and $H'$ be proper subgroups of $G$ such that $G=H\cup (H'+b)$ for some $b\in G$. Then $H=H'$ and consequently $[G:H]=2$.
\end{lemma}
\begin{proof}
Firstly, note that $b\notin H$, otherwise $G=H-b\cup H'=H\cup H'$, which is impossible as $G$ cannot be the union of two proper subgroups. Since $H$ and $H+b$ are disjoint, $H+b\subset H'+b$, then $H\subset H'$. Similarly, $H'\subset H$, so $H=H'$ and $[G:H]=2$. 
\end{proof}

\begin{proof}[Proof of Lemma \ref{Prop1}]
Let us recall that $
\Poles(\wp(\alpha_i z+\beta_i))=\frac{1}{\alpha_i}\Lambda-\frac{\beta_i}{\alpha_i}$. Since Equation \eqref{eq3} is satisfied, we obtain that
\begin{equation}\label{subset3}
\frac{1}{\alpha_i}\Lambda-\frac{\beta_i}{\alpha_i}\subset
\bigcup_{j\neq i}  \left(\frac{1}{\alpha_j}\Lambda-\frac{\beta_j}{\alpha_j}\right).    
\end{equation}
Additionally, there exists $i\in \{1,2,3\}$ such that $\Poles(\wp(\alpha_i z+\beta_i))\cap \Poles(\wp(\alpha_j z+\beta_j))\neq\emptyset$ for each $j\neq i$. Without loss of generality, we assume that $\Poles(\wp(\alpha_1 z+\beta_1))\cap \Poles(\wp(\alpha_j z+\beta_j))\neq\emptyset$, for $j=2,3$. 

Now, we fix $\delta_j\in \Poles(\wp(\alpha_1 z+\beta_1))\cap \Poles(\wp(\alpha_j z+\beta_j))$. Using the fact that $\wp$ is $\Lambda$ periodic, we can rewrite Equation \eqref{subset3} as
\[
\frac{1}{\alpha_1}\Lambda+\delta_2\subset\left(\frac{1}{\alpha_2}\Lambda+\delta_2\right)\cup\left(\frac{1}{\alpha_3}\Lambda+\delta_3\right).
\]
Subtracting $\delta_2$ we obtain:
\[
\frac{1}{\alpha_1}\Lambda\subset\left(\frac{1}{\alpha_2}\Lambda\right)\cup\left(\frac{1}{\alpha_3}\Lambda+\delta\right),
\]
where $\delta=\delta_3-\delta_2$. Due to the fact that $\delta_3\in \Poles(\wp(\alpha_1 z+\beta_1))=\frac{1}{\alpha_3}\Lambda+\delta_2$, we have that $\delta\in \frac{1}{\alpha_1}\Lambda$. As a consequence, we obtain 
\begin{align}\label{groups}
\frac{1}{\alpha_1}\Lambda&=\left(\frac{1}{\alpha_1}\Lambda\cap \frac{1}{\alpha_2}\Lambda\right)\cup\left(\frac{1}{\alpha_1}\Lambda\cap \left(\frac{1}{\alpha_3}\Lambda+\delta\right)\right)\nonumber\\
&=\left(\frac{1}{\alpha_1}\Lambda\cap \frac{1}{\alpha_2}\Lambda\right)\cup\left(\left(\frac{1}{\alpha_3}\Lambda\cap \left(\frac{1}{\alpha_1}\Lambda-\delta\right)\right)+\delta\right)\nonumber\\
&=\left(\frac{1}{\alpha_1}\Lambda\cap \frac{1}{\alpha_2}\Lambda\right)\cup\left(\left(\frac{1}{\alpha_3}\Lambda\cap \frac{1}{\alpha_1}\Lambda\right)+\delta\right).
\end{align}
Assuming that $\bigcap_{j= 1}^3 \Poles(\wp(\alpha_j z+\beta_j))=\emptyset$, we have that $\frac{1}{\alpha_1}\Lambda\not\subset \frac{1}{\alpha_2}\Lambda$ and $\frac{1}{\alpha_1}\Lambda\not\subset \frac{1}{\alpha_3}\Lambda$. Then, applying Lemma \ref{Lemma 2} to Equation \eqref{groups}, we get $\frac{1}{\alpha_2}\Lambda\cap \frac{1}{\alpha_1}=\frac{1}{\alpha_3}\Lambda\cap \frac{1}{\alpha_1}\Lambda$, and we also have that $[\frac{1}{\alpha_1}\Lambda:\frac{1}{\alpha_3}\Lambda\cap \frac{1}{\alpha_1}\Lambda]=2$.

We will work on the following three different cases separately:
\begin{itemize}
    \item [(1)] $\frac{1}{\alpha_2}\Lambda\subset \frac{1}{\alpha_1}\Lambda$ and $\frac{1}{\alpha_3}\Lambda\subset \frac{1}{\alpha_1}\Lambda$,
    \item [(2)] $\frac{1}{\alpha_2}\Lambda\subset \frac{1}{\alpha_1}\Lambda$ and $\frac{1}{\alpha_3}\Lambda\not\subset \frac{1}{\alpha_1}\Lambda$, and
    \item [(3)] $\frac{1}{\alpha_2}\Lambda\not\subset \frac{1}{\alpha_1}\Lambda$ and $\frac{1}{\alpha_3}\Lambda\not\subset \frac{1}{\alpha_1}\Lambda$.
\end{itemize}

\textit{Case 1:} In this case, we have that 
\[
\frac{1}{\alpha_2}\Lambda=\frac{1}{\alpha_2}\Lambda\cap \frac{1}{\alpha_1}=\frac{1}{\alpha_3}\Lambda\cap \frac{1}{\alpha_1}\Lambda=\frac{1}{\alpha_3}\Lambda.
\]
Hence, there exists $u\in \End(E)^{\times}$, such that $\alpha_2=u\alpha_3$. Additionally, the index of the lattices $\frac{1}{\alpha_1}\Lambda$ and $\frac{1}{\alpha_2}\Lambda$ is equal to $2$. Since $\alpha_2$ identifies $E$ with its quotient by $\ker(\alpha_2)$, and by the universal property of quotients, there exists an endomorphism $\lambda$ of degree $2$ such that $\alpha_1=\lambda\alpha_2$. 

\textit{Case 2:} In this case, we have 
\[
\frac{1}{\alpha_2}\Lambda=\frac{1}{\alpha_2}\Lambda\cap \frac{1}{\alpha_1}\Lambda=\frac{1}{\alpha_3}\Lambda\cap \frac{1}{\alpha_1}\Lambda\subset\frac{1}{\alpha_3}\Lambda.
\]
Since $\Poles(\wp(\alpha_1 z+\beta_1))\cap \Poles(\wp(\alpha_3 z+\beta_3))\neq\emptyset$, we can rewrite the inclusion \eqref{subset3} as follows:
\[
\frac{1}{\alpha_3}\Lambda=\left(\frac{1}{\alpha_2}\Lambda\right)\cup\left(\frac{1}{\alpha_2}\Lambda+\delta\right).
\]
Thus $[\frac{1}{\alpha_3}\Lambda: \frac{1}{\alpha_2}\Lambda]=[\frac{1}{\alpha_1}\Lambda: \frac{1}{\alpha_2}\Lambda]=2$, which, as was explained in the previous case, implies that there are two non-associated endomorphisms $\lambda_1$ and $\lambda_2$ of degree $2$ such that $\alpha_1=\lambda_1\alpha_2$ and $\alpha_3=\lambda_2\alpha_2$. 
By Proposition II.2.3.1 from \cite{Silverman2} we have that $\lambda_i=[\frac{1\pm\sqrt{-7}}{2}]$, consequently, $\End(E)\cong\Z[\frac{1+\sqrt{-7}}{2}]$. 

The previous construction defines a morphism $\Phi\colon E\to E^3$ whose image is of the form:
\[
\left\{\left(\left[\frac{1-\sqrt{-7}}{2}\right]P+Q_1,P+Q_2, \left[\frac{1+\sqrt{-7}}{2}\right]P+Q_3\right)\colon P\in E\right\},
\]
and $\Phi$ does not pass through the origin. Additionally, when one coordinate is $e_E$ other coordinate must be $e_E$ as well. However, this arrangement is impossible, since the first and third coordinates become $e_E$ twice, while the second coordinate becomes $e_E$ only when $P=-Q_2$.

\textit{Case 3:}
We begin by applying the previous process to $\frac{1}{\alpha_2}\Lambda$ and $\frac{1}{\alpha_2}\Lambda$, which yields:
\[
M:= \frac{1}{\alpha_1}\Lambda\cap\frac{1}{\alpha_2}\Lambda=\frac{1}{\alpha_1}\Lambda\cap\frac{1}{\alpha_3}\Lambda=\frac{1}{\alpha_2}\Lambda\cap\frac{1}{\alpha_3}\Lambda,
\]
and $[\frac{1}{\alpha_i}\Lambda: M]=2$. In particular, we have that 
\[
\left[\frac{1}{\alpha_i}\Lambda:\frac{2}{\alpha_j}\Lambda\right]=\left[\frac{1}{\alpha_i}\Lambda:M\right]\cdot\left[M:\frac{2}{\alpha_j}\Lambda\right]=4.
\]
Therefore, for every $i$ and $j$ in the set $\{1,2,3\}$ with $i\neq j$, there is $\phi_{ij}\in \mathrm{End}(E)$ of degree $4$ such that $\phi_{ij}\alpha_i=2\alpha_j$. 

Solving the equations $a^2+db^2=4$ for $d<0$ with $d\equiv 1,2 \text{ (mod } 4)$ and $a^2+db^2=16$ for $d<0$ and $d\equiv 3 \text{ (mod } 4)$ in $\Z^3$, we find that the endomorphisms of an elliptic curve defined over $\C$ of degree $4$ are $2u$ for some unit $u$ in $\End(E)$, $\pm[\frac{1\pm\sqrt{-15}}{2}]$, or $\pm[\frac{3\pm\sqrt{-7}}{2}]$. 

Notice that if $\phi_{ij}=2u$ for some $u\in\End(E)^{\times}$, then $u\alpha_i=\alpha_j$, which contradicts the fact that $[\frac{1}{\alpha_i}\Lambda: M]=2$. If non of the $\phi_{ij}=2u$, there exists an $i$ in $\{1,2,3\}$ such that $\phi_{ij}=\pm \phi_{ik}$, implying $\alpha_j=\pm\alpha_k$, which again contradicts $[\frac{1}{\alpha_i}\Lambda: M]=2$.

In conclusion, we have established that the only plausible is Case 1, and in this case, there exists an endomorphism $\lambda$ of degree $2$ such that $\alpha_1=\lambda\alpha_2$ as desired.
\end{proof}

To finish this subsection, we will address Case (d) of Theorem \ref{MainTheorem}. 
\begin{lemma}\label{CM ec}
Let $a\in \C^\times$ and consider the elliptic curve $E_a:y^2=x^3+ax^2-3a^2x+a^3$. The image of the map $\Phi\colon E\to E^3$ defined by $\Phi(P)=(\sqrt{-2}P,\pm P,\pm P+(a,0))$ is contained in $X_{c_1,c_2,c_3}$, where $c_1=2c_2=2c_3$.
\end{lemma}

\begin{proof}
Firstly, observe that $\Phi((a,0))=(e_E, (a,0),e_E)$ and $\Phi(e_E)=(e_E,e_E,(a,0))$. By Lemma \ref{Lemma}, we know that $\Phi(e_E), \Phi((a,0))\in X_{c_1,c_2,c_3}$. Note that $E_a$ is isomorphic over $\C$ to $E:y^2=x^3+4x^2+2x$ via the morphism $(x,y)\mapsto(x/a-1,y/\sqrt{a^3})$. Using Proposition II.2.3.1 from \cite{Silverman2} we obtain an explicit formula for the endomorphism $[\sqrt{-2}]$ in $E_a$:
\[
[\sqrt{-2}](X,Y)=\left(-\frac{1}{2}\left(X+a+\frac{2a^2}{X-a}\right),\frac{-1}{2\sqrt{-2}}Y\left(1-\frac{2a^2}{(X-a)^2}\right)\right),
\]
for $(X,Y)\neq(a,0),e_E$. On the other hand, we have that the $x$-coordinate of $(X,Y)+(a,0)$ is $(aX+a^2)/(X-a)$, whenever $(X,Y)\neq(a,0),e_E$. 

Now, we claim that for $P=(X,Y)\neq (a,0),e_E$ we have
$$
\Phi(P)=([\sqrt{-2}](X,Y), (X,\pm Y), (X,\pm Y)+(a,0))\in X_{c_1,c_2,c_3}.
$$ 
To see this, notice that:
\begin{align*}
2x([\sqrt{-2}](X,Y))+X +x((X,\pm Y)+(a,0))&= -\left(X+a+\frac{2a^2}{X-a}\right) +X +\frac{aX+a^2}{X-a}\\
&=\frac{-a(X-a)-2a^2 +aX+a^2}{X-a}=0,
\end{align*}
which yields the desired result.
\end{proof}

\begin{proof}[Proof of Theorem \ref{MainTheorem}] By Lemmas \ref{Lemma} and \ref{Lemma 1}, we only need to consider morphisms $\Phi$ which map onto every component and whose image is contained in $X_{c_1,c_2,c_3}$.
Let us begin by considering the case where $\Phi$ passes through the origin $(e_E,e_E,e_E)$. Here we can rewrite Equation \eqref{eq1} as:
\begin{equation*}\label{origin equation}
c_1\wp(\alpha_1 z)+c_2\wp(\alpha_2 z)+c_3\wp(\alpha_3 z)=-\frac{A}{3}(c_1+c_2+c_3).
\end{equation*}
In particular, we obtain the inclusion:
\begin{equation*}
\frac{1}{\alpha_1}\Lambda\subset
\frac{1}{\alpha_2}\Lambda\cup \frac{1}{\alpha_3}\Lambda.    
\end{equation*}
Notice that at least two elements from the set  
$$
\left\{\frac{1}{\alpha_1},\frac{1}{\alpha_1}\tau,\frac{1}{\alpha_1}(1+\tau)\right\}
$$ 
belong to $\frac{1}{\alpha_2}\Lambda$ or $\frac{1}{\alpha_3}\Lambda$. Hence, we obtain one of the following inclusions:
\begin{equation*}
\frac{1}{\alpha_1}\Lambda\subset
\frac{1}{\alpha_2}\Lambda\qquad\text{or}\qquad \frac{1}{\alpha_1}\Lambda\subset\frac{1}{\alpha_3}\Lambda.    
\end{equation*}

Hence, $\alpha_1\mid \alpha_2$ or $\alpha_1\mid \alpha_3$. Symmetrically, we also obtain $\alpha_2\mid \alpha_1$ or $\alpha_2\mid \alpha_3$, and $\alpha_3\mid \alpha_1$ or $\alpha_3\mid \alpha_2$. In any case $\alpha_j=u\alpha_i$ for $i\neq j$ and $u\in(\End(E))^{\times}$. Without loss of generality, we assume that $\alpha_2=u\alpha_1$ for some $u\in(\End(E))^{\times}$. 

Given that $\wp$ is a transcendental meromorphic function with infinitely many zeros, we can find infinitely many integers $n$ satisfying the equation:
\[
c_1\alpha_{1}^{2n}+c_2u^{2n}\alpha_{1}^{2n}+c_3\alpha_{3}^{2n}=0, 
\]
Thus, we can express this equation as follows:
\begin{equation}\label{origin}
\frac{c_1+c_2u^{2n}}{-c_3}=\left(\frac{\alpha_3}{\alpha_1}\right)^{2n},
\end{equation}
which is satisfied by infinitely many integers. Note that the right-hand side of Equation \eqref{origin} takes only finitely many values. Therefore, $\alpha_3/\alpha_1=v$ is a root of unity.  

By computing the coefficient of $z^{-2}$ in Equation \eqref{origin} we get 
\[
\frac{c_1}{\alpha_1^2}+\frac{c_2}{u^2\alpha_1^2}+\frac{c_3}{v^2\alpha_1^2}=0.
\]
Therefore, we obtain that the image of $\Phi$ is $(P,uP,vP)$ as in case (c) of Theorem \ref{MainTheorem}.

Finally, we assume that $\Phi$ does not pass through the origin. Since $(e_E,e_E,e_E)\notin \Phi(E)$, the set of common poles of the functions $\wp(\alpha_j z+\beta_j)$ for $j=1,2,3$ is empty.
This allows us to apply Lemma \ref{Prop1}. Therefore, we only need to verify that the following equation has no solution in $c_1,c_2,c_3$:
\[
c_1\wp(\lambda z)+c_2\wp(z)+c_3\wp(uz+\beta)=-\frac{A}{3}(c_1+c_2+c_3),
\]
where $u$ is a unit, $\lambda$ is an endomorphism of degree $2$ and $\beta\in \frac{1}{\lambda}\Lambda$.

Computing the coefficient of $z^{-2}$, we find that $\lambda^2\in \Q$ for some prime $\lambda$ above $2$. Consequently, $\lambda$ must be $\pm\sqrt{-2}$, implying that $\End(E)\cong \Z[\sqrt{-2}]$. In this case, we obtain that $c_1=2c_2=2c_3$, and we obtain the equation:
\begin{equation*}
2\wp(\pm\sqrt{-2}z)+\wp(z)+\wp
(\pm z+\beta)=-\frac{4A}{3},   
\end{equation*}
where $\beta\in \frac{1}{\sqrt{-2}}\Lambda$. 
Therefore, there exists an element $R\in\ker(\sqrt{-2})$ such that $x(R)=\frac{4A}{3}$. 

Let $\epsilon_1$ and $\epsilon_2$ be the other two roots of the cubic polynomial from Equation \eqref{Weierstrass2}. Consider the $2$-torsion points $P_i=(\epsilon_i,0)$ and we claim that $\Phi(P_1)=((\frac{4A}{3},0),P_1,P_2)$. To see this, notice that $[\sqrt{-2}]([\sqrt{-2}]P_i)=[-2]P_i=0$ and $P_1+(\frac{4A}{3},0)=P_2$. Since $E$ has complex multiplication by $\Z[\sqrt{-2}]$, $E$ is isomorphic over $\C$ to $y^2=x^3+4x^2+2x$ via the isomorphism $(x,y)\mapsto (\alpha^2x+\beta,\alpha^3y)$. In this case, we have that $P_1=((-2+\sqrt{-2})\alpha^2+\beta,0)$ and
$$
\Phi(((-2+\sqrt{-2})\alpha^2+\beta,0))= \left((\beta,0),((-2+\sqrt{-2})\alpha^2+\beta,0),((-2-\sqrt{-2})\alpha^2+\beta,0)\right).
$$
Therefore, $\Phi((-2+\sqrt{-2})\alpha^2+\beta,0))\in X_{c_1,c_2,c_3}$ with $c_1=2c_2=2c_3$ if and oly if $\beta=\alpha^2$, thus, 
there exists $a\in \C^{\times}$ such that $E$ is defined by the equation:
\[
y^2=x^3+ax^2-3a^2x+a^3.
\]
This establishes that the image of $\Phi$ is as described in case (d) of Theorem \ref{MainTheorem}. By Lemma \ref{CM ec} the image of $\Phi$ is contained in $X_{c_1,c_2,c_3}$ with $c_1=2c_2=2c_3$.
\end{proof}

\section{Arithmetic consequences}

In this section we bound the number of solutions of Equation \eqref{eq123}. We will apply Theorem 1.1'' of \cite{GGK} which we now state:

\begin{theorem}[Gao--Ge--Kuhne]\label{ggk}
Let $F$ be an algebraically closed field of characteristic zero. Let $A/F$ be an abelian variety of dimension $g\geq 1$ and let $\Gamma$ be a finite rank subgroup of $A(F)$. Let $L$ be an ample line bundle on $A$, let $X\subseteq A$ be a closed irreducible subvariety, and let $U$ be the complement of the Kawamata locus of $A$ in $X$. There exists a constant $c(g,\mathrm{deg}_L(X))$ such that
$$\# U(F)\cap\Gamma\leq c(g,\mathrm{deg}_LX)^{\mathrm{rk}\Gamma+1}.$$
\end{theorem}

Let $E/\C$ be an elliptic curve and $\pi_j\colon E^3\to E$ as in the previous section. We consider the invertible sheaf
$$\mathcal{L}=\mathcal{O}(\pi_1^*e_E+\pi_2^*e_E+\pi_3^*e_E)$$
of $E^3$. As in Subsection 5.5 of \cite{GFP}, it can be proved that

\begin{lemma}
The sheaf $\mathcal{L}$ is ample and symmetric, and $\mathrm{deg}_{\mathcal{L}}X_{c_1,c_2,c_3}\leq 96.$
\end{lemma}

Given a triple $(c_1,c_2,c_3)$ of non-zero complex numbers and the associated surface $X_{c_1,c_2,c_3}\subseteq E^3$, denote by $Z_{c_1,c_2,c_3}$ the Kawamata locus of $X_{c_1,c_2,c_3}$. By our work in Section \ref{sectionmain}, we now know that $Z_{c_1,c_2,c_3}$ is the union of the curves of Theorem \ref{MainTheorem} corresponding to our choice of $(c_1,c_2,c_3)$ and $E$.

We now can prove our main arithmetic result:

\begin{proof}[Proof of Theorem \ref{arithmain}]
We apply here Theorem \ref{ggk} with $F=\mathbb{Q}^{alg}$, $A=E^3$, $L=\mathcal{L}$, $X=X_{c_1,c_2,c_3}$, and $U=X_{c_1,c_2,c_3}\setminus Z_{c_1,c_2,c_3}$. For each choice of $i=1,\ldots,96$, Theorem \ref{ggk} gives us an absolute constant $c(3,i)$. Choosing $D=3\max_{1\leq i\leq 96}c(3,i)$, we obtain
$$\# (X_{c_1,c_2,c_3}\setminus Z_{c_1,c_2,c_3})(K)\cap\Gamma\leq \# (X_{c_1,c_2,c_3}\setminus Z_{c_1,c_2,c_3})(F)\cap \Gamma\leq D^{\mathrm{rk}(\Gamma)+1}.$$
\end{proof}

\begin{proof}[Proof of Corollary \ref{cor1}]
From the conditions on $E$ and the fact that no subsum of $c_1+c_2+c_3$ is equal to zero, the Kawamata locus of $X_{c_1,c_2,c_3}$ is formed by the curves of type (a) of Theorem \ref{MainTheorem}. Applying Theorem \ref{arithmain} the result follows.
\end{proof}

\end{document}